\documentclass[10pt,a4paper]{amsart}
\usepackage{amsmath}
\usepackage{amsfonts}
\usepackage{amsthm}
\usepackage{amssymb}
\usepackage{hyperref}\newcommand\C{\mathbb{C}}
\newcommand\CC{\mathcal{C}}
\newcommand\F{\mathbb{F}}
\newcommand\NN{\mathcal{N}}
\renewcommand\P{\mathbb{P}}
\def\e{\varepsilon}
\DeclareMathOperator*{\Res}{Res}
\newtheorem{theorem}{Theorem}[section]
\newtheorem{lemma}[theorem]{Lemma}
\newtheorem{proposition}[theorem]{Proposition}
\newtheorem*{Mertensconj}{The Mertens Conjecture for Elliptic Curves over Finite Fields}

\newenvironment{deflist}[1]
{\begin{list}{}
	{\itemsep=0pt	\parsep=5pt	\topsep=0pt	\parskip=10pt
	\settowidth{\labelwidth}{\hspace{0.28cm}#1}
	\setlength{\leftmargin}{\labelwidth}
	\addtolength{\leftmargin}{\labelsep}
	}}
{\end{list}}
\begin{document}

\title{On the Mertens Conjecture for Elliptic Curves over Finite Fields}

\author{Peter Humphries}

\address{Department of Mathematics, Princeton University, Princeton, New Jersey 08544, USA}

\email{peterch@math.princeton.edu}



\thanks{This research was partially supported by an Australian Postgraduate Award.}

\begin{abstract}
We introduce an analogue of the Mertens conjecture for elliptic curves over finite fields. Using a result of Waterhouse, we classify the isogeny classes of elliptic curves for which this conjecture holds in terms the size of the finite field and the trace of the Frobenius endomorphism acting on the curve.
\end{abstract}

\maketitle

\section{The Mertens Conjecture}

Let $\mu(n)$ denote the M\"{o}bius function, so that for a positive integer $n$,
	\[\mu(n) = \begin{cases}
1 & \text{if $n = 1$,}	\\
(-1)^t & \text{if $n$ is the product of $t$ distinct primes,}	\\
0 & \text{if $n$ is divisible by a perfect square.}
\end{cases}\]
The Mertens conjecture states that the summatory function of the M\"{o}bius function,
	\[M(x) = \sum_{n \leq x}{\mu(n)},
\]
satisfies the inequality
\begin{equation}\label{Mertensconjecture}
|M(x)| \leq \sqrt{x}
\end{equation}
for all $x \geq 1$. This conjecture stems from the work of Mertens \cite{Mertens}, who in 1897 calculated $M(x)$ from $x = 1$ up to $x = 10\,000$ and arrived at the conjecture \eqref{Mertensconjecture}. Notably, this conjecture implies that all of the nontrivial zeroes of the Riemann zeta function $\zeta(s)$ lie on the line $\Re(s) = 1/2$ (that is, that the Riemann hypothesis is true), and also that all such zeroes are simple.

However, Ingham \cite{Ingham} showed in 1942 that a consequence of the Mertens conjecture is that the imaginary parts of the zeroes of $\zeta(s)$ in the upper half-plane must be linearly dependent over the rational numbers, a relation that seems unlikely; while there is yet to be found strong theoretical evidence for the falsity of such a linear dependence, some limited numerical calculations have failed to find any such linear relations \cite{Bateman}, \cite{Best}. Using methods closely related to the work of Ingham, Odlyzko and te Riele \cite{Odlyzko} disproved the Mertens conjecture in 1984, and in fact showed that
\begin{align*}
\limsup_{x \to \infty} \frac{M(x)}{\sqrt{x}} & > 1.06,	\\
\liminf_{x \to \infty} \frac{M(x)}{\sqrt{x}} & < -1.009.
\end{align*}
These bounds have since been improved to $1.218$ and $-1.229$ respectively by Kotnik and te Riele \cite{Kotnik}, and most recently to $1.6383$ and $-1.6383$ respectively by Best and Trudgian \cite{Best}. It seems likely that
\begin{align*}
\limsup_{x \to \infty} \frac{M(x)}{\sqrt{x}} & = \infty,	\\
\liminf_{x \to \infty} \frac{M(x)}{\sqrt{x}} & = -\infty,
\end{align*}
and, from the work of Ingham \cite{Ingham}, this is known to follow from the assumption of the Riemann hypothesis and the linear independence over the rational numbers of the imaginary parts of the zeroes of $\zeta(s)$ in the upper half-plane.

\section{The Mertens Conjecture for Curves over Finite Fields}

A natural variant of this problem is to formulate an analogue of the Mertens conjecture in the setting of global function fields, that is, for nonsingular projective curves over finite fields. The advantage of this function field setting, as opposed to the classical case, is that the Riemann hypothesis is proved, and the associated zeta functions have only finitely many zeroes. Indeed, when the curve is simply the projective line $\P^1$, so that the associated function field is $\F_q(t)$, the Mertens conjecture is true for trivial reasons, as in this case the summatory function of the M\"{o}bius function is bounded in absolute value by $q$; see \cite[p.\ 5]{Cha}. In this paper, we study the Mertens conjecture in the next-most simple case, namely when the genus of the curve is one, that is to say, the case of elliptic curves over finite fields. Our main result is Theorem \ref{mainellcurvethm}, where we state that it is indeed possible for certain elliptic curves to satisfy a formulation of the Mertens conjecture, and we classify which curves satisfy this conjecture in terms of the size of the finite field $q$ and the trace of the Frobenius endomorphism acting on the elliptic curve.

Let $E$ be an elliptic curve over a finite field $\F_q$ of characteristic $p$. For an effective divisor $N$ of $E$, we define the M\"{o}bius function of $E/\F_q$ to be
	\[\mu_{E/\F_q}(N) = \begin{cases}
1 & \text{if $N$ is the zero divisor,}	\\
(-1)^t & \text{if $N$ is the sum of $t$ distinct prime divisors of $E$,}	\\
0 & \text{if a prime divisor of $E$ divides $N$ with order at least $2$.}
\end{cases}\]
We are interested in the behaviour of the summatory function of the M\"{o}bius function of $E/\F_q$,
	\[M_{E/\F_q}(X) = \sum_{0 \leq \deg(N) \leq X - 1}{\mu_{E/\F_q}(N)},
\]
where $X$ is a positive integer. We wish to determine the validity of the following conjecture.

\begin{Mertensconj}
Let $E$ be an elliptic curve over $\F_q$, and let $M_{E/\F_q}(X)$ be the summatory function of the M\"{o}bius function of $E/\F_q$. Then for all sufficiently large positive integers $X$,
	\[\left|M_{E/\F_q}(X)\right| \leq q^{X / 2}.
\]
\end{Mertensconj}

Note that in the definition of $M_{E/\F_q}(X)$, we are summing over effective divisors $N$ of $E$ for which $0 \leq \deg(N) \leq X - 1$, as opposed to the classical case, where $M(x)$ is a sum over positive integers $n$ for which $1 \leq n \leq x$. It is also noteworthy that the classical form of the Mertens conjecture states that $|M(x)| \leq \sqrt{x}$, whereas our function field version instead states that $\left|M_{E/\F_q}(X)\right| \leq q^{X / 2}$. The reason that $q^X$ replaces $x$ is due to the fact that the absolute norm of an effective divisor $N$ of $E$ is $q^{\deg(N)}$, whereas the absolute norm of a positive integer $n$ is merely $n$ itself. 

Our main result is the following.

\begin{theorem}\label{mainellcurvethm}
Let $E$ be an elliptic curve over a finite field $\F_q$ of characteristic $p$. Then the Mertens conjecture for $E/\F_q$ is true if and only if the order of the finite field $q$ and the trace $a$ of the Frobenius endomorphism acting on $E$ over $\F_q$ satisfy precisely one of the following conditions:
\begin{enumerate}
\item[(1)]	$q = p^m$ with $a = 2$, where either $m$ is arbitrary and $p \neq 2$, or $m = 1$ and $p = 2$,
\item[(2)]	$q = p^m$ with $a = \sqrt{q}$, where $m$ is even and $p \not\equiv 1 \pmod{3}$,
\item[(3)]	$q = p^m$ with $a = 0$, where either $m$ is even and $p \not\equiv 1 \pmod{4}$, or $m$ is odd.
\end{enumerate}
In all these cases, we have that
	\[\left|M_{E/\F_q}(X)\right| \leq q^{X / 2}
\]
for all $X \geq 1$, and also that for every $\e > 0$, there exist infinitely many positive integers $X$, dependent on $q$ and $a$, for which
	\[\left|M_{E/\F_q}(X)\right| > (1 - \e) q^{X / 2}.
\]
\end{theorem}

In a related article \cite{Humphries}, the author studies the Mertens conjecture for higher-genus curves $C$ over finite fields $\F_q$. There is as yet no classification of isogeny classes for curves of a given genus $g$ outside of Waterhouse's classification of elliptic curves \cite{Waterhouse}, as well as the recent classification of Howe, Nart, and Ritzenthaler of curves of genus $2$ \cite{Howe}, so it is no longer possible to determine directly the isogeny classes of curves for which the Mertens conjecture holds. Instead, the author studies the average number of curves satisfying the Mertens conjecture in a particular family of curves over finite fields, and shows that for a curve $C$ in the chosen family, we ought to expect that
	\[\limsup_{X \to \infty} \frac{\left|M_{C/\F_q}(X)\right|}{q^{X / 2}} > 1.
\]

\section{Explicit Expressions for $M_{C\F_q}(X) / q^{X / 2}$}

To study $M_{E/\F_q}(X)$, we must first introduce the zeta function of $E/\F_q$, $Z_{E/\F_q}(u)$. This is defined initially for a complex variable $u$ in the open disc $|u| < q^{-1}$ via the absolutely convergent series
 	\[Z_{E/\F_q}(u) = \exp\left(\sum^{\infty}_{n = 1}{\# E\left(\F_{q^n}\right) \frac{u^n}{n}}\right).
\]
Equivalently, $Z_{E/\F_q}(u) = \zeta_{E/\F_q}(s)$ for $u = q^{-s}$, where
	\[\zeta_{E/\F_q}(s) = \sum_{D \geq 0}{\frac{1}{{\NN D}^s}}.
\]
Here the sum is over all effective divisors $D$ of $E$, and $\NN D = q^{\deg(D)}$ denotes the absolute norm of $D$. This Dirichlet series is absolutely convergent for $\Re(s) > 1$, with an Euler product expansion
	\[\zeta_{E/\F_q}(s) = \prod_{P}{\frac{1}{1 - {\NN P}^{-s}}},
\]
where the product is over all prime divisors $P$ of $E$. This in turn implies that $Z_{E/\F_q}(u)$ is nonvanishing in the open disc $|u| < q^{-1}$. Much more than this is true: $Z_{E/\F_q}(u)$ extends meromorphically to the entire complex plane, satisfies a certain function equation, and also a certain form of the Riemann hypothesis.

\begin{theorem}[see {\cite[Theorems 5.9 and 5.10]{Rosen}}]
Given an elliptic curve $E$ over $\F_q$, there exists a quadratic polynomial $P_{E/\F_q}(u)$ such that for $|u| < q^{-1}$,
\begin{equation}\label{zetafactor}
Z_{E/\F_q}(u) = \frac{P_{E/\F_q}(u)}{(1 - u)(1 - q u)}.
\end{equation}
This yields a meromorphic extension of $Z_{E/\F_q}(u)$ to the whole complex plane, with simple poles at $u = q^{-1}$ and $u = 1$. Furthermore, $Z_{E/\F_q}(u)$ satisfies the functional equation
	\[Z_{E/\F_q}(u) = q^{g - 1} u^{2(g - 1)} Z_{E/\F_q}\left(\frac{1}{q u}\right).
\]
Finally, the polynomial $P_{E/\F_q}(u)$ is of the form
	\[P_{E/\F_q}(u) = 1 - au + qu^2,
\]
with $a$ an integer satisfying $|a| \leq 2 \sqrt{q}$, so that $a = 2 \sqrt{q} \cos \theta$ for some $0 \leq \theta \leq \pi$.
\end{theorem}

The polynomial $P_{E/\F_q}(u) = 1 - au + qu^2$ factorises over $\C$ as
	\[P_{E/\F_q}(u) = \left(1 - \gamma_1 u\right) \left(1 - \gamma_2 u\right)
\]
for some complex numbers $\gamma_1, \gamma_2$ with $\gamma_1 = \overline{\gamma_2}$; we call these the inverse zeroes of $\zeta_{E/\F_q}(s)$. Without loss of generality, we may assume that $\Im(\gamma_1) \geq 0$. As $\gamma_1 + \overline{\gamma_1} = a$ and $\gamma_1 \overline{\gamma_1} = q$, we have that
\begin{align*}
\gamma_1 & = \sqrt{q} e^{i \theta},	\\
\gamma_2 & = \sqrt{q} e^{- i \theta},
\end{align*}
where
\begin{equation}\label{thetaeq}
\theta = \arccos\left(\frac{a}{2 \sqrt{q}}\right).
\end{equation}
Geometrically, the integer $a$ is the trace of the Frobenius endomorphism acting on the elliptic curve $E$ over $\F_q$; the angle $\theta$ is called the Frobenius angle of $E/\F_q$. Notably, there are several restrictions on the possible values that $a$ may take. The following lemma fully characterises the possible values of $a$.

\begin{lemma}[Waterhouse {\cite[Theorem 4.1]{Waterhouse}}]\label{Waterhouselemma}
Let $a$ be an integer. Then $a$ is the trace of the Frobenius endomorphism acting on some elliptic curve $E$ over a finite field $\F_q$ of characteristic $p$ if and only if one of the following conditions is satisfied:
\begin{deflist}{\textup{(4) (iii)}}
\item[\textup{(1)}]	$a \not\equiv 0 \pmod{p}$ and $|a| < 2 \sqrt{q}$; for such an integer $a$, the associated Frobenius angle $\theta$ is such that $\theta / \pi$ is irrational,
\item[\textup{(2) (i)}]	$q = p^m$ with $a = 2 \sqrt{q}$, where $m$ is even, so that $\theta = 0$,
\item[\textup{(2) (ii)}]	$q = p^m$ with $a = - 2 \sqrt{q}$, where $m$ is even, so that $\theta = \pi$,
\item[\textup{(3) (i)}]	$q = p^m$ with $a = \sqrt{q}$, where $m$ is even and $p \not\equiv 1 \pmod{3}$, so that $\theta = \pi / 3$,
\item[\textup{(3) (ii)}]	$q = p^m$ with $a = - \sqrt{q}$, where $m$ is even and $p \not\equiv 1 \pmod{3}$, so that $\theta = 2 \pi / 3$,
\item[\textup{(4) (i)}]	$q = 2^m$ with $a = \sqrt{2 q}$, where $m$ is odd, so that $\theta = \pi / 4$,
\item[\textup{(4) (ii)}]	$q = 2^m$ with $a = - \sqrt{2 q}$, where $m$ is odd, so that $\theta = 3 \pi / 4$,
\item[\textup{(4) (iii)}]	$q = 3^m$ with $a = \sqrt{3 q}$, where $m$ is odd, so that $\theta = \pi / 6$,
\item[\textup{(4) (iv)}]	$q = 3^m$ with $a = - \sqrt{3 q}$, where $m$ is odd, so that $\theta = 5 \pi / 6$,
\item[\textup{(5)}]	$q = p^m$ with $a = 0$, where either $m$ is even and $p \not\equiv 1 \pmod{4}$, or $m$ is odd, so that $\theta = \pi / 2$.
\end{deflist}
That is, there is a bijective correspondence between the isogeny classes of elliptic curves over $\F_q$ and the values of the integer $a$ given in the above conditions.
\end{lemma}

The method of proof of Theorem \ref{mainellcurvethm} involves determining an explicit expression for $M_{E/\F_q}(X)$ in terms of $a$ and $q$ by first studying the Dirichlet series
\begin{equation}\label{muDirichlet}
\sum_{D \geq 0}{\frac{\mu_{E/\F_q}(D)}{{\NN D}^s}}.
\end{equation}
This has previously been done for arbitrary nonsingular projective curves over finite fields by Cha \cite{Cha}, who uses the resulting expression to study the average size of the quantity
	\[\limsup_{X \to \infty} \frac{\left|M_{C/\F_q}(X)\right|}{q^{X / 2}}
\]
when averaged over a particular family of curves $C$, in the limit as the size of the finite field $\F_q$ tends to infinity. Our results are similar and follow the same method, but by restricting ourselves to the case of curves of genus one, we are able to determine exact formul\ae{}, while we also have the advantage of using the classification in Lemma \ref{Waterhouselemma} of the possible values of the trace of the Frobenius endomorphism.

To begin, we note that the M\"{o}bius function of $E/\F_q$ is multiplicative and satisfies $\mu_{E/\F_q}(P) = -1$ and $\mu_{E/\F_q}\left(P^t\right) = 0$ whenever $t \geq 2$  for any prime divisor $P$ of $E$, and so the Dirichlet series \eqref{muDirichlet} has the Euler product
	\[\sum_{D \geq 0}{\frac{\mu_{E/\F_q}(D)}{{\NN D}^s}} = \prod_{P}{\left(1 - {\NN P}^{-s}\right)}
\]
for $\Re(s) > 1$, which, upon comparing Euler products, yields the identity
\begin{equation}\label{muDir}
\sum_{D \geq 0}{\frac{\mu_{E/\F_q}(D)}{{\NN D}^s}} = \frac{1}{\zeta_{E/\F_q}(s)},
\end{equation}
which is valid for all $\Re(s) > 1$. On the other hand,
\begin{equation}\label{Dirmucoeff}
\sum_{D \geq 0}{\frac{\mu_{E/\F_q}(D)}{{\NN D}^s}} = \sum_{D \geq 0}{\frac{\mu_{E/\F_q}(D)}{q^{\deg(D)s}}} = \sum^{\infty}_{N = 0}{\frac{1}{q^{Ns}} \sum_{\deg(D) = N}{\mu_{E/\F_q}(D)}}.
\end{equation}
So determining an expression for the coefficients of the Dirichlet series for $1/\zeta_{E/\F_q}(s)$ using the known factorisation \eqref{zetafactor} of $\zeta_{E/\F_q}(s)$ and then comparing coefficients will lead us to a precise formula for $M_{E/\F_q}(X)$.

\begin{lemma}[cf.\ Cha {\cite[Proposition 2.2]{Cha}}]
For each $N \geq 0$ and any $T > 0$,
\begin{equation}\label{muCauchy}
\frac{1}{2 \pi i} \oint_{\CC_T}{\frac{1}{u^{N + 1}} \frac{1}{Z_{E/\F_q}(u)} \, du} = \sum_{\deg(D) = N}{\mu_{E/\F_q}(D)} + \sum_{\gamma} \Res_{u = \gamma^{-1}}{\frac{1}{u^{N + 1}} \frac{1}{Z_{E/\F_q}(u)}}
\end{equation}
where the sum is over the inverse zeroes $\gamma$ of $Z_{E/\F_q}(u)$, counted without multiplicity, and $\CC_T = \{z \in \C : |z| = q^T\}$. Furthermore, the left-hand side of \eqref{muCauchy} vanishes for $N \geq 1$.
\end{lemma}

\begin{proof}
This is essentially proved in \cite[Proposition 2.2]{Cha} in more generality; the chief difference here is the use of a varying contour. Consider the contour integral
\begin{equation}\label{contourint}
\frac{1}{2\pi i} \oint_{\CC_T}{\frac{1}{u^{N+1}} \frac{1}{Z_{E/\F_q}(u)} \, du},
\end{equation}
where $\CC_T = \{z \in \C : |z| = q^T\}$. We can write $1/Z_{E/\F_q}(u)$ in two ways; via \eqref{zetafactor}, and via \eqref{muDir} and \eqref{Dirmucoeff}, yielding the identities
\begin{align}
\frac{1}{Z_{E/\F_q}(u)} & = \frac{(1 - u) (1 - qu)}{\left(1 - \gamma_1 u\right) \left(1 - \gamma_2 u\right)},	\label{1overZ}	\\
\frac{1}{Z_{E/\F_q}(u)} & = \sum^{\infty}_{N = 0}{u^N \sum_{\deg(D) = N}{\mu_{E/\F_q}(D)}}, \label{1overZmu}
\end{align}
where the first identity is valid for all $u \in \C \setminus \{\gamma_1^{-1}, \gamma_{2}^{-1}\}$, and the second identity is valid for all $|u| < q^{-1}$. So the singularities of the integrand of \eqref{contourint} in the interior of the closed curve $\CC_T$ occur at $u = 0$ and at $u = \gamma^{-1}$ for each zero $\gamma^{-1}$ of $Z_{E/\F_q}(u)$. At the singularity $u = 0$, we have by \eqref{1overZmu} that
	\[\Res_{u = 0} \frac{1}{u^{N+1}} \frac{1}{Z_{E/\F_q}(u)} = \sum_{\deg(D) = N}{\mu_{E/\F_q}(D)}.
\]
The identity \eqref{muCauchy} now follows by Cauchy's residue theorem. Now \eqref{1overZ} and the fact that $|u| = q^T$ and $\left|\gamma_1\right| = \left|\gamma_2\right| = \sqrt{q}$ imply that
\begin{align*}
\left|\frac{1}{2\pi i} \oint_{\CC_T}{\frac{1}{u^{N+1}} \frac{1}{Z_{E/\F_q}(u)} \, du}\right| & \leq \frac{1}{2\pi} \oint_{\CC_T}{\left|\frac{1}{u^{N+1}} \frac{1}{Z_{E/\F_q}(u)}\right| \, |du|}	\\
& \leq \frac{\left(q^T + 1\right) \left(q^{1 + T} + 1\right)}{\left(q^{1/2 + T} - 1\right)^2} q^{-N T}.
\end{align*}
As the right-hand side of \eqref{muCauchy} is independent of $T$, we may take the limit as $T$ tends to infinity in order to find that the contour integral above is zero if $N \geq 1$.
\end{proof}

We are now able to determine an explicit expression for $M_{E/\F_q}(X) / q^{X / 2}$. We must consider two cases: when $Z_{E/\F_q}(u)$ has only simple zeroes, and when $Z_{E/\F_q}(u)$ has a zero of order $2$. For the first case, we have the following result.

\begin{proposition}\label{ellcurveprop1}
Let $E$ be an elliptic curve over $\F_q$, and suppose that $Z_{E/\F_q}(u)$ has only simple zeroes. Then
\begin{equation}\label{Mertensasympa}
\frac{M_{E/\F_q}(X)}{q^{X / 2}} = 2 \sqrt{\frac{q + 1 - a}{4q - a^2}} \cos\left(\omega + X \theta\right),
\end{equation}
where $a$ is the trace of the Frobenius endomorphism, the Frobenius angle $\theta \in [0, \pi]$ is given by \eqref{thetaeq}, and $\omega \in (- \pi / 2, \pi / 2)$ is given by
	\[\omega = \arctan\left(\frac{a - 2}{\sqrt{4q - a^2}}\right).
\]
\end{proposition}

We remark that \eqref{Mertensasympa} is equivalent to
\begin{equation}\label{Mertensasympa2}
\frac{M_{E/\F_q}(X)}{q^{X/2}} = \cos(X \theta) - \frac{a - 2}{\sqrt{4q - a^2}} \sin(X \theta)
\end{equation}
via the cosine angle-sum formula.

\begin{proof}
We write $\gamma$ for $\gamma_1$ and $\overline{\gamma}$ for $\gamma_2$. The fact that $Z_{E/\F_q}(u)$ has only simple zeroes is equivalent to $\gamma \neq \overline{\gamma}$, and hence that $a \neq \pm 2 \sqrt{q}$, and consequently
\begin{align*}
\Res_{u = \gamma^{-1}} \frac{1}{u^{N + 1}} \frac{1}{Z_{E/\F_q}(u)} & = \lim_{u \to \gamma^{-1}} \left(u - \gamma^{-1}\right) \frac{1}{u^{N + 1}} \frac{(1 - u) (1 - qu)}{\left(1 - \gamma u\right) \left(1 - \overline{\gamma} u\right)}	\\
& = - \gamma^N \frac{(\gamma - 1) \left(\overline{\gamma} - 1\right)}{\overline{\gamma} - \gamma}.
\end{align*}
Similarly,
	\[\Res_{u = \overline{\gamma}^{-1}} \frac{1}{u^{N + 1}} \frac{1}{Z_{E/\F_q}(u)} = \overline{\gamma}^N \frac{(\gamma - 1) \left(\overline{\gamma} - 1\right)}{\overline{\gamma} - \gamma}.
\]
It follows that when $N = 0$, the left-hand side of \eqref{muCauchy} is equal to $1$, as the sum over the inverse zeroes $\gamma, \overline{\gamma}$ on the right-hand side of \eqref{muCauchy} vanishes when $N = 0$, whereas
	\[\sum_{\deg(D) = 0}{\mu_{E/\F_q}(D)} = 1,
\]
as the only effective divisor of $E$ of degree zero is the zero divisor. Thus
	\[\sum_{\deg(D) = N}{\mu_{E/\F_q}(D)} = \frac{(\gamma - 1) \left(\overline{\gamma} - 1\right)}{\overline{\gamma} - \gamma} \left(\gamma^N - \overline{\gamma}^N\right) + \begin{cases}
1 & \text{if $N = 0$,}	\\
0 & \text{otherwise}.
\end{cases}\]
Summing this expression from $N = 0$ to $N = X - 1$ and evaluating the resulting geometric series, we find that
	\[M_{E/\F_q}(X) = \frac{\overline{\gamma} - 1}{\overline{\gamma} - \gamma} \gamma^X + \frac{\gamma - 1}{\gamma - \overline{\gamma}} \overline{\gamma}^X,
\]
and hence that
	\[\frac{M_{E/\F_q}(X)}{q^{X/2}} = 2\Re\left(\frac{\overline{\gamma} - 1}{\overline{\gamma} - \gamma} e^{i X \theta}\right) = \frac{\sqrt{q + 1 - 2 \sqrt{q} \cos \theta}}{\sqrt{q} \sin \theta} \cos\left(\omega + X \theta\right),
\]
where
	\[\omega = \arctan\left(\frac{\sqrt{q} \cos \theta - 1}{\sqrt{q} \sin \theta}\right),
\]
and we have used the fact that $\gamma = \sqrt{q} e^{i \theta}$. We complete the proof by noting that
	\[2 \sqrt{q} \sin \theta = \sqrt{4 q - a^2}
\]
as $a = 2 \sqrt{q} \cos \theta$ with $0 \leq \theta \leq \pi$.
\end{proof}

We also have the following analogous result in the case where $Z_{E/\F_q}(u)$ has a zero of multiple order.

\begin{proposition}\label{ellcurveprop2}
Let $E$ be an elliptic curve over a finite field $\F_q$ of characteristic $p$, and suppose that $Z_{E/\F_q}(u)$ has zeroes of multiple order, so that $q = p^m$ with $a = \pm 2 \sqrt{q}$, where $m$ is even. Then
\begin{equation}\label{Mertensasympa3}
\frac{M_{E/\F_q}(X)}{q^{X / 2}} = - (\pm 1)^X \left(1 \mp \frac{1}{\sqrt{q}}\right) X + (\pm 1)^X.
\end{equation}
\end{proposition}

\begin{proof}
If $a = \pm 2 \sqrt{q}$, then $\gamma = \overline{\gamma} = \pm \sqrt{q}$. Now
\begin{align*}
\Res_{u = \pm q^{- 1 / 2}}{\frac{1}{u^{N + 1}} \frac{1}{Z_{E/\F_q}(u)}} & = \lim_{u \to \pm q^{- 1 / 2}} \frac{d}{du} \frac{\left(u \mp q^{- 1 / 2}\right)^2}{u^{N + 1}} \frac{(1 - u)(1 - qu)}{\left(1 \mp \sqrt{q} u\right)^2}	\\
& = (\pm 1)^{N + 1} \left(\sqrt{q} \mp 1\right)^2 N q^{(N - 1) / 2}.
\end{align*}
As this vanishes when $N = 0$, we must again have that
	\[\frac{1}{2 \pi i} \oint_{\CC_T}{\frac{1}{u} \frac{1}{Z_{E/\F_q}(u)} \, du} = 1
\]
via \eqref{muCauchy}, and consequently
	\[\sum_{\deg(D) = N}{\mu_{E/\F_q}(D)} = - (\pm 1)^{N + 1} \left(\sqrt{q} \mp 1\right)^2 N q^{(N - 1) / 2} + \begin{cases}
1 & \text{if $N = 0$,}	\\
0 & \text{otherwise,}
\end{cases}\]
which leads to the result upon summing over all $0 \leq N \leq X - 1$ and then dividing through by $q^{X / 2}$.
\end{proof}

\section{Proof of Theorem \ref{mainellcurvethm}}

Using Propositions \ref{ellcurveprop1} and \ref{ellcurveprop2}, we are now able to determine the quantity
	\[\limsup_{X \to \infty} \frac{\left|M_{E/\F_q}(X)\right|}{q^{X / 2}}
\]
for each elliptic curve $E$ over a given finite field $\F_q$. We must consider each possible combination of values for $q$ and $a$ as determined in Lemma \ref{Waterhouselemma}, which will culminate in a proof of Theorem \ref{mainellcurvethm}. 

\begin{deflist}{(4) (iii)}
\item[(1)]	If $q = p^m$ with $a \not\equiv 0 \pmod{p}$ and $|a| < 2 \sqrt{q}$, then by \eqref{Mertensasympa},
	\[\left|M_{E/\F_q}(X)\right| \leq 2 \sqrt{\frac{q + 1 - a}{4q - a^2}} q^{X / 2}
\]
for all $X \geq 1$. As the Frobenius angle $\theta$ is such that $\theta/\pi$ is irrational, the Weyl equidistribution theorem implies that $X \theta$ is equidistributed modulo $\pi$ as $X$ tends to infinity, and hence
	\[\limsup_{X \to \infty} \frac{\left|M_{E/\F_q}(X)\right|}{q^{X / 2}} = 2 \sqrt{\frac{q + 1 - a}{4q - a^2}} = \sqrt{1 + \frac{(a - 2)^2}{4q - a^2}}.
\]
So the Mertens conjecture for $E/\F_q$ is true precisely when the inequality
	\[\sqrt{1 + \frac{(a - 2)^2}{4q - a^2}} \leq 1
\]
holds, which can only occur when $a = 2$, provided that $p \neq 2$. Note that even in this case, we nevertheless have via the Weyl equidistribution theorem that for every $\e > 0$, there exist infinitely many values of $X$, dependent on $q$, for which
	\[\left|M_{E/\F_q}(X)\right| > (1 - \e) q^{X / 2}.
\]
\item[(2)] If $q = p^m$ with $a = \pm 2 \sqrt{q}$, where $m$ is even, then from \eqref{Mertensasympa3},
	\[\limsup_{X \to \infty} \frac{\left|M_{E/\F_q}(X)\right|}{q^{X / 2}} = \infty.
\]
\item[(3) (i)] If $q = p^m$ with $a = \sqrt{q}$, where $m$ is even and $p \not\equiv 1 \pmod{3}$, we have from \eqref{Mertensasympa2} that
	\[\frac{M_{E/\F_q}(X)}{q^{X / 2}} = \cos\left(\frac{\pi X}{3}\right) - \frac{\sqrt{3}}{3}\left(1 - \frac{2}{\sqrt{q}}\right) \sin\left(\frac{\pi X}{3}\right).
\]
We calculate the six cases of $X \pmod{6}$: \vspace{0.2cm}

\begin{center}
\begin{tabular}{cc}
\hline
$X \pmod{6}$ & $M_{E/\F_q}(X) / q^{X / 2}$	\\	\hline
$0$ & $1$	\\
$1$ & $1 / \sqrt{q}$	\\
$2$ & $- 1 + 1 / \sqrt{q}$	\\
$3$ & $- 1$	\\
$4$ & $- 1 / \sqrt{q}$	\\
$5$ & $1 - 1 / \sqrt{q}$	\\ \hline
\end{tabular}
\end{center} \vspace{0.2cm}

\noindent So for all $X \geq 1$,
	\[\left|M_{E/\F_q}(X)\right| \leq q^{X / 2},
\]
with equality occurring infinitely often.
\item[(3) (ii)] Similarly, if $q = p^m$ with $a = - \sqrt{q}$, where $m$ is even and $p \not\equiv 1 \pmod{3}$,
	\[\frac{M_{E/\F_q}(X)}{q^{X / 2}} = \cos\left(\frac{2 \pi X}{3}\right) + \frac{\sqrt{3}}{3}\left(1 + \frac{2}{\sqrt{q}}\right) \sin\left(\frac{2 \pi X}{3}\right).
\]
The three cases of $X \pmod{3}$ are \vspace{0.2cm}

\begin{center}
\begin{tabular}{cc}
\hline
$X \pmod{3}$ & $M_{E/\F_q}(X) / q^{X / 2}$	\\	\hline
$0$ & $1$	\\
$1$ & $1 / \sqrt{q}$	\\
$2$ & $- 1 - 1 / \sqrt{q}$	\\ \hline
\end{tabular}
\end{center} \vspace{0.2cm}

\noindent This shows that
	\[\limsup_{X \to \infty} \frac{\left|M_{E/\F_q}(X)\right|}{q^{X / 2}} = 1 + \frac{1}{\sqrt{q}}.
\]
\item[(4) (i)] If $q = 2^m$ with $a = \sqrt{2 q}$, where $m$ is odd, then
	\[\frac{M_{E/\F_{2^m}}(X)}{2^{m X / 2}} = \cos\left(\frac{\pi X}{4}\right) - \left(1 - \frac{1}{2^{(m - 1) / 2}}\right) \sin\left(\frac{\pi X}{4}\right).
\]
We analyse the eight cases of $X \pmod{8}$: \vspace{0.2cm}

\begin{center}
\begin{tabular}{cc}
\hline
$X \pmod{8}$ & $M_{E/\F_{2^m}}(X) / 2^{m X / 2}$	\\	\hline
$0$ & $1$	\\
$1$ & $2^{- m / 2}$	\\
$2$ & $- 1 + 2^{- (m - 1) / 2}$	\\
$3$ & $- \sqrt{2} + 2^{- m / 2}$	\\
$4$ & $- 1$	\\
$5$ & $- 2^{- m / 2}$	\\
$6$ & $1 - 2^{- (m - 1) / 2}$	\\
$7$ & $\sqrt{2} - 2^{- m / 2}$	\\ \hline
\end{tabular}
\end{center} \vspace{0.2cm}

\noindent So when $m = 1$, we have that
	\[\left|M_{E/\F_2}(X)\right| \leq q^{X / 2}
\]
for all $X \geq 1$, with equality occurring infinitely often, while for $m \geq 3$,
	\[\limsup_{X \to \infty} \frac{\left|M_{E/\F_{2^m}}(X)\right|}{2^{m X / 2}} = \sqrt{2} - \frac{1}{2^{m / 2}}.
\]
\item[(4) (ii)] Likewise, if $q = 2^m$ with $a = - \sqrt{2 q}$, where $m$ is odd, then
	\[\frac{M_{E/\F_{2^m}}(X)}{2^{m X / 2}} = \cos\left(\frac{3 \pi X}{4}\right) + \left(1 + \frac{1}{2^{(m - 1) / 2}}\right) \sin\left(\frac{3 \pi X}{4}\right).
\]
The table of values of $X \pmod{8}$ is \vspace{0.2cm}

\begin{center}
\begin{tabular}{cc}
\hline
$X \pmod{8}$ & $M_{E/\F_{2^m}}(X) / 2^{m X / 2}$	\\	\hline
$0$ & $1$	\\
$1$ & $2^{- m / 2}$	\\
$2$ & $1 + 2^{- (m - 1) / 2}$	\\
$3$ & $\sqrt{2} + 2^{- m / 2}$	\\
$4$ & $- 1$	\\
$5$ & $- 2^{- m / 2}$	\\
$6$ & $ - 1 - 2^{- (m - 1) / 2}$	\\
$7$ & $- \sqrt{2} - 2^{- m / 2}$	\\ \hline
\end{tabular}
\end{center} \vspace{0.2cm}

\noindent Thus
	\[\limsup_{X \to \infty} \frac{\left|M_{E/\F_{2^m}}(X)\right|}{2^{m X / 2}} = \sqrt{2} + \frac{1}{2^{m / 2}}.
\]
\item[(4) (iii)] If $q = 3^m$ with $a = \sqrt{3 q}$, where $m$ is odd, then
	\[\frac{M_{E/\F_{3^m}}(X)}{3^{m X / 2}} = \cos\left(\frac{\pi X}{6}\right) - \left(1 - \frac{2}{3^{m / 2}}\right) \sin\left(\frac{\pi X}{6}\right).
\]
The twelve cases of $X \pmod{12}$ are \vspace{0.2cm}

\begin{center}
\begin{tabular}{cc}
\hline
$X \pmod{12}$ & $M_{E/\F_{3^m}}(X) / 3^{m X / 2}$	\\	\hline
$0$ & $1$	\\
$1$ & $(\sqrt{3} - 1) / 2 + 3^{- m / 2}$	\\
$2$ & $- (\sqrt{3} - 1) / 2 + 3^{- (m - 1) / 2}$	\\
$3$ & $- 1 + 2 \times 3^{- m / 2}$	\\
$4$ & $- (\sqrt{3} + 1) / 2 - 3^{- (m - 1) / 2}$	\\
$5$ & $- (\sqrt{3} + 1) / 2 + 3^{- m / 2}$	\\
$6$ & $- 1$	\\
$7$ & $- (\sqrt{3} - 1) / 2 - 3^{- m / 2}$	\\
$8$ & $(\sqrt{3} - 1) / 2 - 3^{- (m - 1) / 2}$	\\
$9$ & $1 - 2 \times 3^{- m / 2}$	\\
$10$ & $(\sqrt{3} + 1) / 2 + 3^{- (m - 1) / 2}$	\\
$11$ & $(\sqrt{3} + 1) / 2 - 3^{- m / 2}$	\\ \hline
\end{tabular}
\end{center} \vspace{0.2cm}

\noindent Consequently,
	\[\limsup_{X \to \infty} \frac{\left|M_{E/\F_{3^m}}(X)\right|}{3^{m X / 2}} = \frac{\sqrt{3} + 1}{2} + \frac{1}{3^{(m - 1) / 2}}.
\]
\item[(4) (iv)] Next, if $q = 3^m$ with $a = - \sqrt{3 q}$, where $m$ is odd, then
	\[\frac{M_{E/\F_{3^m}}(X)}{3^{m X / 2}} = \cos\left(\frac{5 \pi X}{6}\right) + \left(1 + \frac{2}{3^{m / 2}}\right) \sin\left(\frac{5 \pi X}{6}\right).
\]
Now we have the table \vspace{0.2cm}

\begin{center}
\begin{tabular}{cc}
\hline
$X \pmod{12}$ & $M_{E/\F_{3^m}}(X) / 3^{m X / 2}$	\\	\hline
$0$ & $1$	\\
$1$ & $- (\sqrt{3} - 1) / 2 + 3^{- m / 2}$	\\
$2$ & $- (\sqrt{3} - 1) / 2 - 3^{- (m - 1) / 2}$	\\
$3$ & $1 + 2 \times 3^{- m / 2}$	\\
$4$ & $- (\sqrt{3} + 1) / 2 - 3^{- (m - 1) / 2}$	\\
$5$ & $(\sqrt{3} + 1) / 2 - 3^{- m / 2}$	\\
$6$ & $- 1$	\\
$7$ & $(\sqrt{3} - 1) / 2 - 3^{- m / 2}$	\\
$8$ & $(\sqrt{3} - 1) / 2 + 3^{- (m - 1) / 2}$	\\
$9$ & $- 1 - 2 \times 3^{- m / 2}$	\\
$10$ & $(\sqrt{3} + 1) / 2 + 3^{- (m - 1) / 2}$	\\
$11$ & $- (\sqrt{3} + 1) / 2 + 3^{- m / 2}$	\\ \hline
\end{tabular}
\end{center} \vspace{0.2cm}

\noindent So we have that
	\[\limsup_{X \to \infty} \frac{\left|M_{E/\F_{3^m}}(X)\right|}{3^{m X / 2}} = \frac{\sqrt{3} + 1}{2} + \frac{1}{3^{(m - 1) / 2}}.
\]
\item[(5)] Finally, if $q = p^m$ with $a = 0$, where either $m$ is even and $p \not\equiv 1 \pmod{4}$, or $m$ is odd, then
	\[\frac{M_{E/\F_q}(X)}{q^{X / 2}} = \cos\left(\frac{\pi X}{2}\right) + \frac{1}{\sqrt{q}} \sin\left(\frac{\pi X}{2}\right).
\]
The four cases of $X \pmod{4}$ are \vspace{0.2cm}

\begin{center}
\begin{tabular}{cc}
\hline
$X \pmod{4}$ & $M_{E/\F_q}(X) / q^{X / 2}$	\\	\hline
$0$ & $1$	\\
$1$ & $1 / \sqrt{q}$	\\
$2$ & $- 1$	\\
$3$ & $- 1 / \sqrt{q}$	\\ \hline
\end{tabular}
\end{center} \vspace{0.2cm}

\noindent Thus we have the inequality
	\[\left|M_{E/\F_q}(X)\right| \leq q^{X / 2}
\]
for all $X \geq 1$, with equality occurring infinitely often.
\end{deflist}

\subsection*{Acknowledgments}
The author thanks Jim Borger for his helpful advice and corrections of this paper, and also Mike Zieve for alerting the author to the results in \cite{Howe}.

\end{document}